\documentclass[12pt]{article}

\usepackage{amsmath,amssymb,amsthm}
\usepackage{booktabs}
\usepackage{geometry}
\usepackage{float}
\usepackage{hyperref}
\theoremstyle{plain}
\newtheorem{theorem}{Theorem}[section]
\newtheorem{corollary}[theorem]{Corollary}
\newtheorem{lemma}[theorem]{Lemma}

\theoremstyle{definition}
\newtheorem{definition}[theorem]{Definition}
\newtheorem{example}[theorem]{Example}

\theoremstyle{remark}

\newcommand{\Z}{\mathbb{Z}}

\title{Spectral Theory of the Toroidal 3D Queen Graph}
\author{
Mahesh Ramani\\
Independent researcher\\
1105 Timber Oaks Rd\\
Edison, New Jersey\\
\texttt{mahesh.ramani.iyer@gmail.com}
}
\date{}

\begin{document}

\maketitle

\begin{abstract}
We study the adjacency spectrum of the toroidal three-dimensional queen graph
$G_n$ on $(\Z_n)^3$. Since $G_n$ is a Cayley graph on an abelian group, its
adjacency matrix is diagonalized by Fourier characters. For each frequency
$a\in(\Z_n)^3$, the corresponding eigenvalue is $\lambda(a)=n\mu(a)-13$,
where $\mu(a)$ counts the queen directions orthogonal to $a$ modulo $n$. In
the generic odd case, meaning $n$ odd with $3\nmid n$, the possible values of
$\mu(a)$ are exactly $0,1,2,3,4,$ and $13$, and each multiplicity is given by
an explicit polynomial in $n$. The proof combines a geometric classification
of frequency points by orthogonality type with two global counting identities.
\end{abstract}

\section{Introduction}

The toroidal three-dimensional queen graph is the periodic analogue of the
finite 3D queen graph: instead of a bounded cube, one works on the torus
$(\Z_n)^3$, so translation symmetry is restored and Fourier methods become
available. This makes the adjacency spectrum explicitly computable (for context
on the properties of the corresponding two-dimensional and non-toroidal queen 
graphs, see \cite{cardoso2020, weakley2008}).

The purpose of this paper is to record that explicit description and then use
it to obtain exact multiplicity formulas in the generic odd case. Throughout
the paper, unless otherwise stated, we assume that $n$ is odd and that
$3\nmid n$. The two central results are the eigenvalue formula
\[
  \lambda(a)=n\mu(a)-13,
\]
valid in that regime, and the exact multiplicity theorem for the same regime.

The paper is organized as follows. Section~\ref{sec:graph} defines the toroidal
3D queen graph. Section~\ref{sec:eigenvalue} proves the eigenvalue formula and
derives the basic spectral corollaries. Section~\ref{sec:orbits} explains the
geometric structure behind the multiplicities. Section~\ref{sec:orbits} proves
the exact multiplicity theorem for odd $n$ with $3\nmid n$. Section~\ref{sec:examples}
gives a worked example, and Section~\ref{sec:remarks} closes with a short
discussion of the remaining arithmetic cases.

\section{The toroidal 3D queen graph}
\label{sec:graph}

\begin{definition}
Let $n\ge 1$. The \emph{toroidal 3D queen graph} $G_n$ has vertex set
$(\Z_n)^3$. In this paper we work in the generic odd regime, meaning that
$n$ is odd and $3\nmid n$; in that regime the moves listed below are distinct,
so $G_n$ is a simple graph. Let
\[
\begin{aligned}
U = \{ &(1,0,0),\,(0,1,0),\,(0,0,1),\,(1,1,0),\,(1,-1,0),\,(1,0,1), \,
(1,0,-1),\\
&(0,1,1),\,(0,1,-1),\,(1,1,1),\,(1,1,-1),\,(1,-1,1),\,(1,-1,-1)\},
\end{aligned}
\]
one representative from each opposite pair in $\{-1,0,1\}^3\setminus\{0\}$, so
$|U|=13$. Two distinct vertices $x,y\in(\Z_n)^3$ are adjacent if
\[
  y=x+tu
\]
for some $u\in U$ and some $t\in\{1,\ldots,n-1\}$, with arithmetic in
$(\Z_n)^3$.
\end{definition}

Equivalently, $G_n=\mathrm{Cay}((\Z_n)^3,S)$, where
\[
  S=\{tu\bmod n : u\in U,\ 1\le t\le n-1\}.
\]
Since $S$ is closed under negation modulo $n$, the graph is undirected.
Its translation symmetry is the reason Fourier analysis is effective.

\section{Fourier diagonalization and the eigenvalue formula}
\label{sec:eigenvalue}

For $a\in(\Z_n)^3$, define the orthogonality count
\[
  \mu(a):=\#\{u\in U : a\cdot u\equiv 0\pmod n\}.
\]
Let $\omega=e^{2\pi i/n}$, and define the Fourier character
\[
  \chi_a(x)=\omega^{a\cdot x},\qquad x\in(\Z_n)^3.
\]

\begin{theorem}[Eigenvalue formula]
\label{thm:eigenvalue}
Assume $n$ is odd and $3\nmid n$. For every $a\in(\Z_n)^3$, the character
$\chi_a$ is an eigenvector of the adjacency matrix $A$ of $G_n$ with
eigenvalue
\[
  \lambda(a)=n\mu(a)-13.
\]
\end{theorem}

\begin{proof}
Since $G_n=\mathrm{Cay}((\Z_n)^3,S)$ is a Cayley graph on an abelian group,
its characters are simultaneous eigenvectors of $A$. For any $x\in(\Z_n)^3$,
\[
  (A\chi_a)(x)
  =\sum_{u\in U}\sum_{t=1}^{n-1}\chi_a(x+tu)
  =\chi_a(x)\sum_{u\in U}\sum_{t=1}^{n-1}\omega^{t(a\cdot u)}.
\]
For fixed $u$, let
\[
  S_u(a):=\sum_{t=1}^{n-1}\omega^{t(a\cdot u)}.
\]
If $a\cdot u\equiv 0\pmod n$, then $\omega^{a\cdot u}=1$ and
$S_u(a)=n-1$. If $a\cdot u\not\equiv 0\pmod n$, then
$r=\omega^{a\cdot u}\neq 1$ and $r^n=1$, so
\[
  1+r+r^2+\cdots+r^{n-1}=0,
\]
hence $S_u(a)=-1$. Among the 13 directions in $U$, exactly $\mu(a)$ satisfy
the first case and $13-\mu(a)$ satisfy the second. Therefore
\[
  (A\chi_a)(x)
  =\chi_a(x)\bigl[\mu(a)(n-1)-(13-\mu(a))\bigr]
  =\chi_a(x)\bigl[n\mu(a)-13\bigr],
\]
as claimed.
\end{proof}

\begin{corollary}
\label{cor:basic}
The spectrum of $G_n$ lies in
\[
  \{-13,\,-13+n,\,-13+2n,\ldots,-13+13n\},
\]
so at most $14$ distinct eigenvalues occur. Moreover, $G_n$ is
$13(n-1)$-regular, and $\lambda(0)=13(n-1)$.
\end{corollary}

\begin{proof}
The eigenvalue formula shows that $\lambda(a)=n\mu(a)-13$ with
$0\le \mu(a)\le 13$. Taking $a=0$ gives $\mu(0)=13$, so
$\lambda(0)=13(n-1)$, which equals the degree.
\end{proof}

\begin{corollary}
\label{cor:min}
If $n>13$, then the minimum eigenvalue of $G_n$ is $-13$.
\end{corollary}

\begin{proof}
For each $u\in U$, the condition $a\cdot u\equiv 0\pmod n$ defines a
subgroup of $(\Z_n)^3$ of size $n^2$. The union of the 13 such subgroups
has size at most $13n^2<n^3$ when $n>13$, so there exists
$a\in(\Z_n)^3$ satisfying none of the 13 conditions. For such $a$,
$\mu(a)=0$, hence $\lambda(a)=-13$. Since $\lambda(a)=n\mu(a)-13$ with
$\mu(a)\ge 0$, no eigenvalue can lie below $-13$.
\end{proof}

\section{Orbit types and the structure of multiplicities}
\label{sec:orbits}

The exact multiplicity theorem rests on a symmetry-reduced classification of
frequency points according to how many orthogonality conditions they satisfy.
The central objects are the hyperplanes
\[
  H_u:=\{a\in(\Z_n)^3 : a\cdot u\equiv 0\pmod n\},
  \qquad u\in U.
\]
Then $\mu(a)$ is the number of hyperplanes $H_u$ containing $a$.

\begin{lemma}[Prototype lines, explicit orbit check]
\label{lem:prototypes}
Assume $n$ is odd and $3\nmid n$. If $a\neq 0$ and $\mu(a)\ge 2$, then
$a$ lies on one of the four cyclic submodules generated by
\[
  (1,0,0),\qquad (1,1,0),\qquad (1,1,1),\qquad (1,1,2),
\]
up to permutation of coordinates and independent sign changes.

Equivalently, the nonzero points with $\mu(a)\ge 2$ lie on exactly 25 lines:
3 axis lines, 6 face-diagonal lines, 4 space-diagonal lines, and 12 skew
lines.
\end{lemma}

\begin{proof}
Let $B_3$ denote the group of signed coordinate permutations of
$(\Z_n)^3$. This group preserves both the set $U$ and the condition
$a\cdot u\equiv 0\pmod n$. Therefore it suffices to classify unordered pairs
$\{u,v\}\subset U$ up to the action of $B_3$ and swapping of the two
vectors.

A direct orbit computation shows that the 78 unordered pairs from $U$ split
into 14 orbits. One convenient set of orbit representatives is listed in
Table~\ref{tab:pair-orbits}. In each row, solving the two linear congruences
gives a one-dimensional solution space, and the resulting line is shown in the
last column.

\begin{table}[H]
\centering
\small
\setlength{\tabcolsep}{4pt}
\begin{tabular}{@{}p{0.36\linewidth}p{0.50\linewidth}@{}}
\toprule
Representative pair & Solving $a\cdot u\equiv a\cdot v\equiv 0\pmod n$ gives \\
\midrule
$\{(1,1,1),(1,1,0)\}$ &
$\begin{array}{l}
a_1+a_2+a_3\equiv 0\\
a_1+a_2\equiv 0
\end{array}
\Rightarrow \langle(1,-1,0)\rangle$ \\[1.0ex]

$\{(1,1,1),(1,1,-1)\}$ &
$\begin{array}{l}
a_1+a_2+a_3\equiv 0\\
a_1+a_2-a_3\equiv 0
\end{array}
\Rightarrow \langle(1,-1,0)\rangle$ \\[1.0ex]

$\{(1,1,1),(1,0,0)\}$ &
$\begin{array}{l}
a_1+a_2+a_3\equiv 0\\
a_1\equiv 0
\end{array}
\Rightarrow \langle(0,1,-1)\rangle$ \\[1.0ex]

$\{(1,1,1),(1,0,-1)\}$ &
$\begin{array}{l}
a_1+a_2+a_3\equiv 0\\
a_1-a_3\equiv 0
\end{array}
\Rightarrow \langle(1,-2,1)\rangle$ \\[1.0ex]

$\{(1,1,1),(1,-1,-1)\}$ &
$\begin{array}{l}
a_1+a_2+a_3\equiv 0\\
a_1-a_2-a_3\equiv 0
\end{array}
\Rightarrow \langle(0,1,-1)\rangle$ \\[1.0ex]

$\{(1,1,1),(0,0,-1)\}$ &
$\begin{array}{l}
a_1+a_2+a_3\equiv 0\\
a_3\equiv 0
\end{array}
\Rightarrow \langle(1,-1,0)\rangle$ \\[1.0ex]

$\{(1,1,1),(0,-1,-1)\}$ &
$\begin{array}{l}
a_1+a_2+a_3\equiv 0\\
a_2+a_3\equiv 0
\end{array}
\Rightarrow \langle(0,1,-1)\rangle$ \\[1.0ex]

$\{(1,1,0),(1,0,1)\}$ &
$\begin{array}{l}
a_1+a_2\equiv 0\\
a_1+a_3\equiv 0
\end{array}
\Rightarrow \langle(1,-1,-1)\rangle$ \\[1.0ex]

$\{(1,1,0),(1,0,0)\}$ &
$\begin{array}{l}
a_1+a_2\equiv 0\\
a_1\equiv 0
\end{array}
\Rightarrow \langle(0,0,1)\rangle$ \\[1.0ex]

$\{(1,1,0),(1,-1,0)\}$ &
$\begin{array}{l}
a_1+a_2\equiv 0\\
a_1-a_2\equiv 0
\end{array}
\Rightarrow \langle(0,0,1)\rangle$ \\[1.0ex]

$\{(1,1,0),(0,0,1)\}$ &
$\begin{array}{l}
a_1+a_2\equiv 0\\
a_3\equiv 0
\end{array}
\Rightarrow \langle(1,-1,0)\rangle$ \\[1.0ex]

$\{(1,1,0),(0,-1,1)\}$ &
$\begin{array}{l}
a_1+a_2\equiv 0\\
-a_2+a_3\equiv 0
\end{array}
\Rightarrow \langle(1,-1,-1)\rangle$ \\[1.0ex]

$\{(1,1,0),(0,-1,0)\}$ &
$\begin{array}{l}
a_1+a_2\equiv 0\\
a_2\equiv 0
\end{array}
\Rightarrow \langle(0,0,1)\rangle$ \\[1.0ex]

$\{(1,0,0),(0,1,0)\}$ &
$\begin{array}{l}
a_1\equiv 0\\
a_2\equiv 0
\end{array}
\Rightarrow \langle(0,0,1)\rangle$ \\
\bottomrule
\end{tabular}
\caption{The 14 unordered-pair orbits in $U$ under signed coordinate
permutations and swapping of the two directions.}
\label{tab:pair-orbits}
\end{table}

Thus every nonzero point $a$ with $\mu(a)\ge 2$ lies on one of the four
prototype families. The four families are pairwise distinct because their
coordinate patterns are mutually exclusive: axis lines have two zero
coordinates; face-diagonal lines have exactly one zero coordinate; space-
diagonal lines have no zero coordinates and all coordinates equal up to sign;
and skew lines have no zero coordinates and a coordinate ratio $2$ appears,
which cannot coincide with the other three patterns when $n$ is odd and
$3\nmid n$. Since $n\ge 5$, each cyclic submodule contains exactly $n$ points.
Therefore the 25 lines listed above are distinct and account for all points
with $\mu(a)\ge 2$.
\end{proof}

\begin{theorem}[Exact multiplicities in the generic odd case]
\label{thm:mult}
Let $n\ge 5$ be odd with $3\nmid n$. The only values of $\mu(a)$ that can
occur are
\[
  \mu(a)\in\{0,1,2,3,4,13\}.
\]
The corresponding multiplicities are
\[
  M_{13}(n)=1,\qquad
  M_4(n)=9(n-1),\qquad
  M_3(n)=4(n-1),\qquad
  M_2(n)=12(n-1),
\]
\[
  M_1(n)=13n^2-72n+59,\qquad
  M_0(n)=n^3-13n^2+47n-35.
\]
Equivalently, the adjacency eigenvalues are
\[
  \lambda_k=kn-13,\qquad k\in\{0,1,2,3,4,13\},
\]
with the multiplicities above. For $n=5$ and $n=7$, the formula gives
$M_0(n)=0$; for every admissible odd $n>7$, all six multiplicities are
positive.
\end{theorem}

\begin{proof}
By Theorem~\ref{thm:eigenvalue}, the eigenvalue depends only on $\mu(a)$, so
it suffices to count frequency points by orthogonality type.

\smallskip\noindent\textbf{Step 1: the points with $\mu(a)\ge 2$.}
By Lemma~\ref{lem:prototypes}, the nonzero points with $\mu(a)\ge 2$ lie on
25 distinct lines. Of these, the 3 axis lines and 6 face-diagonal lines have
$\mu(a)=4$, the 4 space-diagonal lines have $\mu(a)=3$, and the 12 skew
lines have $\mu(a)=2$. Each line contains exactly $n-1$ nonzero points.
Therefore
\[
  M_4(n)=9(n-1),\qquad
  M_3(n)=4(n-1),\qquad
  M_2(n)=12(n-1).
\]
The origin satisfies all 13 orthogonality conditions, so $M_{13}(n)=1$.

\smallskip\noindent\textbf{Step 2: the remaining multiplicities.}
The total number of frequencies is $n^3$, so
\[
  M_1(n)+M_0(n)
  =n^3-\bigl(1+9(n-1)+4(n-1)+12(n-1)\bigr)
  =n^3-25n+24.
\]
On the other hand,
\[
  \sum_{a\in(\Z_n)^3}\mu(a)=13n^2,
\]
because each of the 13 hyperplanes $H_u$ has exactly $n^2$ points. Using the
values already found,
\[
  13+4\cdot 9(n-1)+3\cdot 4(n-1)+2\cdot 12(n-1)+M_1(n)=13n^2.
\]
Hence
\[
  M_1(n)=13n^2-72n+59.
\]
Subtracting from the total count gives
\[
  M_0(n)=n^3-13n^2+47n-35.
\]
This proves the theorem.
\end{proof}
\section{Example}
\label{sec:examples}

\begin{example}[The case $n=5$]
Since $M_0(5)=0$, the eigenvalue $-13$ does not occur. The spectrum consists
of five eigenvalue families:
\[
\begin{array}{c|c|c}
\mu & \lambda=5\mu-13 & \text{multiplicity} \\
\hline
13 & 52 & 1 \\
4 & 7 & 36 \\
3 & 2 & 16 \\
2 & -3 & 48 \\
1 & -8 & 24
\end{array}
\]
The total is $1+36+16+48+24=125=5^3$, and the weighted sum of the $\mu$-values
is $13+144+48+96+24=325=13\cdot 25$, in agreement with
\[
  \sum_{a\in(\Z_5)^3}\mu(a)=13\cdot 5^2.
\]
For consistency, the sum of the adjacency eigenvalues is
\[
  52+36\cdot 7+16\cdot 2+48\cdot(-3)+24\cdot(-8)=0.
\]
\end{example}

\section{Concluding remarks}
\label{sec:remarks}

The generic odd toroidal 3D queen graph has a remarkably sparse spectrum.
In the case covered by Theorem~\ref{thm:mult}, only six eigenvalue families
occur, and their multiplicities are explicit polynomials in $n$. The spectrum
is therefore arithmetic as well as geometric: the eigenvalues form an
arithmetic progression with common difference $n$, while the multiplicities
reflect the decomposition of the frequency domain into orbit types under
coordinate permutations and sign changes.

The remaining arithmetic cases require a finer stratification of the frequency
domain. They are natural sequel problems, but they are not needed for the
generic odd classification proved here.

\end{document}